\documentclass[12 pt]{amsart}
\usepackage{amssymb,bezier}
\usepackage[all]{xy}
\newtheorem{theorem}{Theorem}[section]
\newtheorem{definition}[theorem]{Definition}

\newtheorem{remark}[theorem]{Remark}

\newtheorem{lemma}[theorem]{Lemma}

\begin{document}
\baselineskip 17pt
\bibliographystyle{alpha}

\title{A Finite Dimensional $A_{\infty}$ Algebra Example}

\author{Michael P. Allocca}
\address{Department of 
Mathematics, North Carolina State University,
Raleigh NC 27695, USA}
\email{mpallocc@math.ncsu.edu}
\author{Tom Lada}
\address{Department of Mathematics, North Carolina State University,
Raleigh NC 27695, USA}
\email{lada@math.ncsu.edu}

\begin{abstract}
We construct an example of an $A_{\infty}$ algebra structure defined over a finite dimensional graded vector space.
\end{abstract}

\dedicatory{Dedicated to Tornike Kadeishvili on the occasion of his 60th birthday}

\maketitle

\parskip 4pt

\section*{Introduction}

$A_{\infty}$ algebras (or sha algebras) and $L_{\infty}$ (or sh Lie algebras) have been topics of current research.  Construction of small examples of these algebras can play a role in gaining insight into deeper properties of these structures.  These examples may prove useful in developing a deformation theory as well as a representation theory for these algebras.

In \cite{DL}, an $L_\infty$ algebra structure on the graded vector space $V=V_0\oplus V_1$ where $V_0$ is a 2 dimensional vector space, and $V_1$ is a 1 dimensional space, is discussed.  This surprisingly rich structure on this small graded vector space was shown by Kadeishvili and Lada, \cite{KL}, to be an example of an open-closed homotopy algebra (OCHA) defined by Kajiura and Stasheff \cite{KS}.  In an unpublished note \cite{D} M. Daily constructs a variety of other $L_\infty$ algebra structures  on this same vector space.   

In this article we add to this collection of structures on the vector space $V$ by providing a detailed construction of non-trivial $A_\infty$ algebra data for $V$.

\section{$A_{\infty}$ Algebras}
We first recall the definition of an $A_{\infty}$ algebra (Stasheff \cite{JS}). 

\begin{definition}\label{defA}
 Let $V$ be a graded vector space.  An $A_{\infty}$ structure on $V$ is a collection of linear maps $m_k: V^{\otimes k} \rightarrow V$ of degree $2-k$ that satisfy the identity
 $$ \displaystyle\sum_{\lambda=0}^{n-1} \displaystyle\sum_{k=1}^{n-\lambda}  \alpha m_{n-k+1}(x_1 \otimes \cdots \otimes x_{\lambda} \otimes  m_k(x_{\lambda +1} \otimes \cdots \otimes x_{\lambda + k}) \otimes x_{\lambda + k +1} \otimes \cdots \otimes x_n) =0$$
where $\alpha = (-1)^{k + \lambda + k \lambda + kn + k(|x_1| + \cdots + |x_{\lambda}|)}$, for all $n \geq 1$.\\
\end{definition}

This utilizes the cochain complex convention.  One may alternatively utilize the chain complex convention by requiring each map $m_k$ to have degree $k-2$.

We will define the desuspension of V (denoted $\downarrow\! V$) as the graded vector space with indices given by $(\downarrow\! V)_n = V_{n+1}$, and the desuspension operator, $\downarrow: V \rightarrow (\downarrow\! V)$ (resp. suspension operator $\uparrow: (\downarrow V) \rightarrow V$) in the natural sense.   We will also employ the usual Koszul sign convention in this setting:  whenever two symbols (objects or maps) of degree $p$ and $q$ are commuted, a factor of $(-1)^{pq}$ is introduced.  Subsequently, $\uparrow ^{\otimes n} \circ \downarrow ^{\otimes n} = (-1)^{\frac{n(n-1)}{2}} id$ and $\downarrow\! x_1 \otimes \downarrow\! x_2 \otimes \cdots \otimes \downarrow\! x_n = (-1)^{\sum_{i=1}^{n}(n-i)|x_i|} \downarrow\!^{\otimes n} (x_1 \otimes x_2 \otimes \cdots \otimes x_n)$.

Stasheff also showed that an $A_{\infty}$ structure on $V$ is equivalent to the existence of a degree 1 coderivation $D: T^*(\downarrow\! V) \rightarrow T^*(\downarrow\! V)$ with the property $D^2 = 0$. 
Here, $T^*(\downarrow\! V)$ is the tensor coalgebra on the graded vector space $\downarrow\! V$.  Such a coderivation is constructed by defining $m_k^{\prime}: (\downarrow\! V^{\otimes k}) \rightarrow \downarrow\! V$ by $m_k^{\prime}  = (-1)^{\frac{k(k-1)}{2}} \downarrow\! \circ m_k \circ \uparrow ^{\otimes k}$ and then extending each $m_k^{\prime}$ to a coderivation on $T^*(\downarrow\! V)$.  By ``abuse of notation", $m_k^{\prime}$ can be described by
\begin{align*}
&m_k^{\prime} (\downarrow\! x_1 \otimes \downarrow\! x_2 \otimes \cdots \otimes \downarrow\! x_n )= \sum\limits_{i=0}^{n-1} (1^{\otimes i} \otimes m_k^{\prime} \otimes 1^{\otimes n-i-1})(\downarrow\! x_1 \otimes \downarrow\! x_2 \otimes \cdots \otimes \downarrow\! x_n)\\
&\! = \! \sum\limits_{i=0}^{n-1} (-1)^{(k-2) \!(|x_1| + \cdots +|x_i|-i)} (\downarrow\! x_1 \otimes \!\cdots \!\otimes \downarrow\! x_i \otimes m_k^{\prime}(\downarrow\! x_{i+1} \otimes \!\cdots \!\otimes \downarrow\! x_{i+k} ) \otimes \downarrow\! x_{i+k+1} \otimes \!\cdots \! \otimes \downarrow\! x_n)
\end{align*}
We then define $D := \displaystyle\sum_{k=1}^{\infty} m_k^{\prime}$.

\section{A Finite Dimensional Example}

Let $V$ denote the graded vector space given by $V= \bigoplus V_n$ where $V_0$ has basis  $<v_1,v_2>$, $V_1$ has basis $ <w>$, and $V_n = 0$ for $n \neq 0,1$.   Define a structure on $V$ by the following linear maps $m_n:  V^{\otimes n} \rightarrow V$:
\begin{align*}
m_1(v_1) =  m_1 (v_2) &= w\\
For \; n \geq 2:\;\;
m_n(v_1 \otimes w^{\otimes k} \otimes v_1 \otimes w^{\otimes (n-2)-k})  &= (-1)^k s_n v_1 , \; 0 \leq k \leq n-2\\
m_n(v_1 \otimes w^{\otimes (n-2)} \otimes v_2) &= s_{n+1} v_1\\
m_n(v_1 \otimes w^{\otimes(n-1)}) &= s_{n+1} w
\end{align*}

where $s_n = (-1)^{\frac{(n+1)(n+2)}{2}}$, and $m_n= 0$ when evaluated on any element of $V^{\otimes n}$ that is not listed above.  It is worth noting that this assumes the cochain convention regarding $A_{\infty}$ algebra structures.  Hence $|v_1|=|v_2|=0$ and $|w|=1$.\\

\begin{theorem}\label{thm}
The maps defined above give the graded vector space $V$  an $A_{\infty}$ algebra structure.  
\end{theorem}

The proof of this theorem relies on two lemmas:

\begin{lemma}\label{lem1}
Let $m_n^{\prime} : = (-1)^{\frac{n(n-1)}{2}} \downarrow\! \circ m_n \circ \uparrow ^{\otimes n} : \;(\downarrow\! V)^{\otimes n} \rightarrow \downarrow\! V$.  Under the preceding definitions for $m_n$ and $V$, we obtain the following formulas for $m_n^{\prime}$:
\begin{align*}
m_1^{\prime} &= \downarrow\! m_1\\
For\; n\geq 2:\;\; m_n^{\prime}(\downarrow\! v_1 \otimes (\downarrow\! w)^{\otimes k} \otimes \downarrow\! v_1 \otimes (\downarrow\! w)^{\otimes (n-2)-k}) &= \downarrow\! v_1,\;  0 \leq k \leq n-2\\
m_n^{\prime}(\downarrow\! v_1 \otimes (\downarrow\! w)^{\otimes (n-2)} \otimes \downarrow\! v_2) &= \downarrow\! v_1\\
m_n^{\prime}(\downarrow\! v_1 \otimes (\downarrow\! w)^{\otimes (n-1)}) &= \downarrow\! w
\end{align*}

\end{lemma}

\begin{remark}
Each $m_n^{\prime}$ is of degree 1.
\end{remark}

\begin{proof}[Proof of Lemma \ref{lem1}]
$m_1^{\prime}(\downarrow\!x) = (-1)^0 \downarrow\! \circ m_1 \circ \uparrow(\downarrow\! x) = \downarrow\! m_1(x)$ for any $x$.\\

Now let $n \geq 2$.  The majority of the work here is centered around computing the signs associated with the graded setting.  The elements $x_i$ and the maps $\uparrow$, $\downarrow$, and $m_n$ all contribute to an overall sign via their degrees.  Observing these signs, we find $$m_n^{\prime} (\downarrow\! x_1 \otimes \downarrow\! x_2 \otimes \cdots \otimes \downarrow\! x_n) =(-1)^{\frac{n(n-1)}{2}} \downarrow\! \circ m_n \circ \uparrow ^{\otimes n} (\downarrow\! x_1 \otimes \downarrow\! x_2 \otimes \cdots \otimes \downarrow\! x_n) $$ 
\begin{equation*}
=
\begin{cases}
(-1)^{\sum_{i=1}^{n/2} |x_{2i-1}|} \downarrow\! m_n(x_1 \otimes x_2 \otimes \cdots \otimes x_n) & \text{if n is even.}
\\
(-1)^{\sum_{i=1}^{(n-1)/2} |x_{2i}|} \downarrow\! m_n(x_1 \otimes x_2 \otimes \cdots \otimes x_n) & \text{if n is odd.}
\end{cases}
\end{equation*}

First consider $m_n^{\prime}(\downarrow\! v_1 \otimes (\downarrow\! w)^{\otimes k} \otimes \downarrow\! v_1 \otimes (\downarrow\! w)^{\otimes (n-2)-k}),\;  0 \leq k \leq n-2$.  This computation may be divided into 4 cases based on the parity of $n$ and $k$.  If $n$ and $k$ are both even, then: 
\begin{align*}
m_n^{\prime}(\downarrow\! v_1 \otimes (\downarrow\! w)^{\otimes k} \otimes \downarrow\! v_1 \otimes (\downarrow\! w)^{\otimes (n-2)-k}) & = (-1)^{|v_1| + (\frac{n}{2} -1)|w|} \downarrow\! m_n(v_1 \otimes w^{\otimes k} \otimes v_1 \otimes w^{\otimes (n-2)-k})\\
& = (-1)^{0 + \frac{n}{2} -1} (-1)^k s_n \downarrow\! v_1 \\
& = (-1)^{\frac{n}{2} -1} (-1)^{\frac{(n+1)(n+2)}{2}} \downarrow\! v_1 \\
& = (-1)^{\frac{n}{2} -1} (-1)^{(n+1)(\frac{n}{2}+1)} \downarrow\! v_1\;\; (*)
\end{align*}
If $\frac{n}{2}$ is even, then $(*) = (-1)^{\text{odd}} (-1)^{\text{odd*odd}} \downarrow\! v_1 = \downarrow\! v_1$ where `odd' denotes an odd number.\\
If $\frac{n}{2}$ is odd, then $(*) = (-1)^{\text{even}} (-1)^{\text{odd*even}} \downarrow\! v_1 = \downarrow\! v_1$ where `even' denotes an even number.\\

A similar argument holds in the remaining 3 cases.  Hence $$ m_n^{\prime}(\downarrow\! v_1 \otimes (\downarrow\! w)^{\otimes k} \otimes \downarrow\! v_1 \otimes (\downarrow\! w)^{\otimes (n-2)-k}) = \downarrow\! v_1,\;  0 \leq k \leq n-2$$

Now consider $m_n^{\prime}(\downarrow\! v_1 \otimes (\downarrow\! w)^{\otimes (n-2)} \otimes \downarrow\! v_2)$.  This computation may be divided into 2 cases based on the parity of $n$.  If $n$ is even, then: 

\begin{align*}
m_n^{\prime}(\downarrow\! v_1 \otimes (\downarrow\! w)^{\otimes (n-2)} \otimes \downarrow\! v_2) &= (-1)^{|v_1| + (\frac{n}{2}-1)|w|} m_n(\downarrow\! v_1 \otimes (\downarrow\! w)^{\otimes (n-2)} \otimes \downarrow\! v_2)\\
&= (-1)^{\frac{n}{2}-1} s_{n+1} \downarrow\! v_1\\
& = (-1)^{\frac{n}{2}-1} (-1)^{\frac{(n+2)(n+3)}{2}} \downarrow\! v_1\\
&= (-1)^{\frac{n}{2}-1} (-1)^{(\frac{n}{2}-1)(n+3)} \downarrow\! v_1 \;\; (*)
\end{align*}
If $\frac{n}{2}$ is even, then $(*) = (-1)^{\text{odd}} (-1)^{\text{odd*odd}} \downarrow\! v_1 = \downarrow\! v_1$ where `odd' denotes an odd number.\\
If $\frac{n}{2}$ is odd, then $(*) = (-1)^{\text{even}} (-1)^{\text{even*odd}} \downarrow\! v_1 = \downarrow\! v_1$ where `even' denotes an even number.\\

A similar argument holds in the case that $n$ is odd.  Hence $$m_n^{\prime}(\downarrow\! v_1 \otimes (\downarrow\! w)^{\otimes (n-2)} \otimes \downarrow\! v_2) = \downarrow\! v_1$$

The preceding argument may also be repeated for $m_n^{\prime}(\downarrow\! v_1 \otimes (\downarrow\! w)^{\otimes (n-1)})$. Hence $$m_n^{\prime}(\downarrow\! v_1 \otimes (\downarrow\! w)^{\otimes (n-1)}) = \downarrow\! w$$
\end{proof}
\medskip

\begin{lemma}\label{lem2}
Let $D = \displaystyle\sum_{k=1}^{\infty} m_k^{\prime}$ where $m_k^{\prime}$ is defined in Lemma \ref{lem1}.  Let $n \geq 2$ be a positive integer.  Suppose $D^2 (\downarrow\! x_1 \otimes \downarrow\! x_2 \otimes \cdots \otimes  \downarrow\! x_{m}) = 0 \; \forall \; x_i  \in   V$, $1\leq m \leq n-1$.    \\

Then $D^2 (\downarrow\! x_1 \otimes \downarrow\! x_2 \otimes \cdots \otimes \downarrow\! x_{n}) = \displaystyle\sum_{i+j = n+1} m_i^{\prime} m_j^{\prime} (\downarrow\! x_1 \otimes \downarrow\! x_2 \otimes \cdots \otimes \downarrow\! x_{n})$
\end{lemma}

\begin{proof}
We first note that 
$$D^2 (\downarrow\! x_1 \otimes \downarrow\! x_2 \otimes \cdots \otimes \downarrow\! x_{n}) = \displaystyle\sum_{i+j\leq n+1} m_i^{\prime} m_j^{\prime} (\downarrow\! x_1 \otimes \downarrow\! x_2 \otimes \cdots \otimes \downarrow\! x_{n})$$
 since $m_k^{\prime} (\downarrow\! x_1 \otimes \downarrow\! x_2 \otimes \cdots \otimes \downarrow\! x_l) = 0$ for $k>l$.  So
 
\begin{align*} 
D^2 (\downarrow\! x_1 \otimes \downarrow\! x_2 \otimes \cdots \otimes \downarrow\! x_{n}) =& \displaystyle\sum_{i+j\leq n} m_i^{\prime} m_j^{\prime} (\downarrow\! x_1 \otimes \downarrow\! x_2 \otimes \cdots \otimes \downarrow\! x_{n}) \\
+& \displaystyle\sum_{i+j = n+1} m_i^{\prime} m_j^{\prime} (\downarrow\! x_1 \otimes \downarrow\! x_2 \otimes \cdots \otimes \downarrow\! x_{n})
\end{align*}

Hence it suffices to show that  $ \displaystyle\sum_{i+j\leq n} m_i^{\prime} m_j^{\prime} (\downarrow\! x_1 \otimes \downarrow\! x_2 \otimes \cdots \otimes \downarrow\! x_{n}) =0$
 
Consider $\displaystyle\sum_{i+j\leq n} m_i^{\prime} m_j^{\prime} (\downarrow\! x_1 \otimes \downarrow\! x_2 \otimes \cdots \otimes \downarrow\! x_{n}) $:  Since $i+j \leq n$, we can break this sum up into 4 different types of of elements in $(\downarrow\! V)^{\otimes k} $ based on whether the first and last terms in the tensor product contain $m_i^{\prime}$ or $m_j^{\prime}$:

\noindent $\bullet$ Type 1: Elements with first term $\downarrow\! x_1$ and last term $\downarrow\! x_n$ \\
\indent (example: $\downarrow\! x_1 \otimes \downarrow\! x_2 \otimes m_1'(\downarrow\! x_3) \otimes m_2'(\downarrow\! x_4 \otimes \downarrow\! x_5) \otimes \downarrow\! x_6$)\\
\noindent $\bullet$ Type 2: Elements with first term $\downarrow\! x_1$ and last term containing $m_k'$ for some $k$ \\
\indent (example: $\downarrow\! x_1 \otimes \downarrow\! x_2 \otimes m_3'(\downarrow\! x_3 \otimes m_2'(\downarrow\! x_4 \otimes \downarrow\! x_5) \otimes \downarrow\! x_6)$)\\
\noindent $\bullet$ Type 3: Elements with first term containing $m_k'$ for some $k$ and last term $\downarrow\! x_n$\\
\indent (example: $m_2'(\downarrow\! x_1 \otimes \downarrow\! x_2) \otimes m_1'(\downarrow\! x_3) \otimes \downarrow\! x_4 \otimes \downarrow\! x_5 \otimes \downarrow\! x_6$)\\
\noindent $\bullet$ Type 4: Elements with first term containing $m_k'$ and last term containing $m_l'$ for some $k,l$ \\
\indent (example: $m_2'(\downarrow\! x_1 \otimes \downarrow\! x_2) \otimes \downarrow\! x_3 \otimes \downarrow\! x_4 \otimes m_2'(\downarrow\! x_5 \otimes \downarrow\! x_6)$)\\ 

Now each term of type 1 must be produced by $m_i' m_j'$ with $i+j \leq n-1$.  Hence, by factorization of tensor products, all possible terms of type 1 are given by: 
\begin{align*}
(-1)^{2|x_1|-2} \Big{(} \downarrow\! x_1 \otimes \big{(} \displaystyle\sum_{i+j\leq n-1} m_i^{\prime} m_j^{\prime} (\downarrow\! x_2 \otimes \downarrow\! x_3 \otimes \cdots \otimes \downarrow\! x_{n-1}) \big{)} \Big{)} \otimes \downarrow\! x_n\\
= \; \Big{(} \downarrow\! x_1 \otimes \big{(} D^2 (\downarrow\! x_2 \otimes \downarrow\! x_3 \otimes \cdots \otimes \downarrow\! x_{n-1}) \big{)} \Big{)} \otimes \downarrow\! x_n\\
=\; \big{(} \downarrow\! x_1 \otimes 0 \big{)} \; \otimes \downarrow\! x_n\\
=\; 0
\end{align*}

\noindent since $D^2 = 0$ when evaluated on $n-2$ terms.  A similar argument holds for the type 2 and type 3 summands.  \\

We now consider type 4 terms.  Consider an arbitrary element of type 4: 
$$m_i'(\downarrow\! x_1 \otimes \cdots \otimes \downarrow\! x_i) \otimes \downarrow\! x_{i+1} \otimes \cdots \otimes \downarrow\! x_{n-j} \otimes m_j'(\downarrow\! x_{n-j+1} \otimes \cdots \otimes \downarrow\! x_n)$$

\noindent Consider how this arbitrary element is generated:  We begin with
$$m_i' m_j' (\downarrow\! x_1 \otimes \cdots \otimes \downarrow\! x_n)$$
We then apply $m_j'$ to the last $j$ terms, which yields:
$$ (-1)^{|x_1| + \cdots + |x_{n-j}| - (n-j)} m_i ' (\downarrow\! x_1 \otimes \cdots \otimes \downarrow\! x_{n-j} \otimes m_j'(\downarrow\! x_{n-j+1} \otimes \cdots \otimes \downarrow\! x_n))$$
Finally we apply $m_i'$ to the first $i$ terms:
$$(-1)^{|x_1| + \cdots + |x_{n-j}| - (n-j)} m_i'(\downarrow\! x_1 \otimes \cdots \otimes \downarrow\! x_i) \otimes \downarrow\! x_{i+1} \otimes \cdots \\ \cdots \otimes \downarrow\! x_{n-j} \otimes m_j'(\downarrow\! x_{n-j+1} \otimes \cdots \otimes \downarrow\! x_n) \;\; (*)$$

Each of these arbitrary type 4 elements can be paired up with an element generated by $m_j'm_i'$ as follows:  Begin with
$$m_j' m_i' (\downarrow\! x_1 \otimes \cdots \otimes \downarrow\! x_n) $$
Then apply $m_i'$ to the first $i$ terms:
$$m_j'(m_i'(\downarrow\! x_1 \otimes \cdots \otimes \downarrow\! x_i) \otimes \downarrow\! x_{i+1} \otimes \cdots \otimes \downarrow\! x_n) $$ 
Finally, apply $m_j'$ to the last $j$ terms:
$$(-1)^{|x_1| + \cdots + |x_{n-j}| - (n-j) +1}  m_i'(\downarrow\! x_1 \otimes \cdots \otimes \downarrow\! x_i) \otimes \downarrow\! x_{i+1} \otimes \cdots \\ \cdots \otimes \downarrow\! x_{n-j} \otimes m_j'(\downarrow\! x_{n-j+1} \otimes \cdots \otimes \downarrow\! x_n) \;\; (**) $$

Since these type 4 elements were arbitrary, and $(*) + (**) = 0$, all type 4 terms added together equal 0.  Hence, all type 1, 2, 3, and 4 terms yield 0, and so\\
$$ \displaystyle\sum_{i+j\leq n} m_i^{\prime} m_j^{\prime} (\downarrow\! x_1 \otimes \downarrow\! x_2 \otimes \cdots \otimes \downarrow\! x_{n}) =0$$
\end{proof}
\medskip

\begin{proof}[Proof of Theorem \ref{thm}]
It is clear that each map $m_n$ is of degree $2-n$.  To prove that these maps yield an $A_{\infty}$ structure, one may verify that they satisfy the identity given in definition \ref{defA}. However, this is a rather daunting task, due to the varying signs, $s_n$, accompanying the $m_n$ maps.  To utilize an alternative method of proof, we construct a degree 1 coderivation, $D$, as described in section 1. 

In the context of Theorem \ref{thm}, we may use the definition for $m_k^{\prime}$ given by Lemma \ref{lem1} to construct $D$.  It then suffices to show that $D^2 =0$. \\

We aim to prove $D^2=0$ by induction on the number of inputs for $D$.  It is worth first noting that  $D = \displaystyle\sum_{k=1}^{\infty} m_k^{\prime}$,  however $D(\downarrow\! x_1 \otimes \cdots \otimes \downarrow\! x_n) = \displaystyle\sum_{k=1}^{n} m_k^{\prime} (\downarrow\! x_1 \otimes \cdots \otimes \downarrow\! x_n)$ since $m_k^{\prime} (\downarrow\! x_1 \otimes \cdots\otimes  \downarrow\! x_n) = 0$ for $k \geq n$.\\

For $n=1$, we have $D^2(\downarrow\! x) = m_1^{\prime}m_1^{\prime} (\downarrow\! x)= \downarrow\! m_1^2(x) = 0 \; \forall \;x\; \in V$.  

Now assume $D^2(\downarrow\! x_1 \otimes \cdots  \otimes\downarrow\! x_{n-1}) = 0$.  We aim to show that  $D^2(\downarrow\! x_1 \otimes \cdots  \otimes\downarrow\! x_{n}) = 0$:\\

\begin{remark}
Since $m_i^{\prime}$ and $m_j^{\prime}$ are linear, it is sufficient to show that $D^2=0$ on only basis elements.
\end{remark}

By Lemma \ref{lem2}, $D^2(\downarrow\! x_1 \otimes \cdots  \otimes\downarrow\! x_{n}) = \displaystyle\sum_{i+j=n+1} m_i^{\prime}m_j^{\prime} (\downarrow\! x_1 \otimes \cdots  \otimes\downarrow\! x_{n})$, hence it suffices to show that $\displaystyle\sum_{i+j=n+1} m_i^{\prime}m_j^{\prime} (\downarrow\! x_1 \otimes \cdots  \otimes\downarrow\! x_{n})=0,\; \forall \; x_1 \cdots x_n \; \in \; V$.\\

It is advantageous to approach this problem from the bottom up, since $x_1 \cdots x_n \; \in \; V$ implies calculating $3^n$ different combinations of elements.  That is, we consider only nontrivial (nonzero) elements in the sum $\displaystyle\sum_{i+j=n+1} m_i^{\prime}m_j^{\prime} (\downarrow\! x_1 \otimes \cdots  \otimes\downarrow\! x_{n})$.  Now since $i+j=n+1$, we observe that $ m_i^{\prime}m_j^{\prime} (\downarrow\! x_1 \otimes \cdots  \otimes\downarrow\! x_{n})\; \in \; (\downarrow\! V)^{\otimes 1}$.  Since, by definition, $m_i^{\prime}$ cannot produce the element $\downarrow\! v_2$, the seemingly large task of considering nontrivial $ m_i^{\prime}m_j^{\prime} (\downarrow\! x_1 \otimes \cdots  \otimes\downarrow\! x_{n})$ yields only two possibilities:
\begin{center}
$m_i^{\prime}m_j^{\prime} (\downarrow\! x_1 \otimes \cdots  \otimes\downarrow\! x_{n}) = c\downarrow\! v_1$\\
\text{   or   }
$m_i^{\prime}m_j^{\prime} (\downarrow\! x_1 \otimes \cdots  \otimes\downarrow\! x_{n}) = c \downarrow\! w \text{ for some constant, c.}$
\end{center}

Therefore if $m_i^{\prime}m_j^{\prime} (\downarrow\! x_1 \otimes \cdots  \otimes\downarrow\! x_{n})\neq 0 \;  \text{for some} \; i+j=n+1$, then $\displaystyle\sum_{i+j=n+1} m_i^{\prime}m_j^{\prime} (\downarrow\! x_1 \otimes \cdots  \otimes\downarrow\! x_{n})$ is a sum of $\downarrow\! v_1$'s or $\downarrow\! w$'s.\\

We first consider the manner in which $m_i^{\prime}m_j^{\prime} (\downarrow\! x_1 \otimes \cdots  \otimes\downarrow\! x_{n})$ yields a $\downarrow\! w$:\\

By defintion of $m_n^{\prime}$, $\downarrow\! w$ must be produced by $m_i^{\prime}(\downarrow\! v_1 \otimes (\downarrow\! w)^{\otimes (i-1)}) \;(*)$.  To accomplish this, the original arrangement $\downarrow\! x_1 \otimes \cdots  \otimes\downarrow\! x_{n})$ must satisfy $x_1=v_1$ and must contain exactly one more `$v$' ($v=v_1$ or $v_2$).\\

\noindent $\bullet$ \textbf{Case 1: $v=v_1$}.  Let us consider $m_i^{\prime}m_j^{\prime}(\downarrow\! v_1 \otimes (\downarrow\! w)^{\otimes k} \otimes \downarrow\! v_1 \otimes (\downarrow\! w)^{\otimes (n-2)-k}),\; 0 \leq k \leq n-2$.\\
Now, to produce $(*)$, $m_j^{\prime}$ must `catch' $(1)$ both $\downarrow\! v_1$'s, or $(2)$ only the second $\downarrow\! v_1$.\\

\noindent $(1)$ We have $m_j^{\prime}(\downarrow\! v_1 \otimes (\downarrow\! w)^{\otimes k} \otimes \downarrow\! v_1 \otimes (\downarrow\! w)^{\otimes (n-2)-k}) = \downarrow\! v_1,\; k+2 \leq j \leq n.$\\
This yields $m_i^{\prime}(\downarrow\! v_1 \otimes (\downarrow\! w) ^{\otimes (n-j)}) = \downarrow\! w$.  Now since $k+2 \leq j \leq n$, there are $n-(k+2)+1 = n-k-1$ such terms in  $\displaystyle\sum_{i+j=n+1} m_i^{\prime}m_j^{\prime} (\downarrow\! v_1 \otimes (\downarrow\! w)^{\otimes k} \otimes \downarrow\! v_1 \otimes (\downarrow\! w)^{\otimes (n-2)-k})$.\\

\noindent $(2)$ We have $(-1)^{|v_1| + k|w| -(k+1)} m_i^{\prime}\Big(\!\downarrow\! v_1 \otimes (\downarrow\! w)^{\otimes k} \otimes \Big[ m_j^{\prime} (\downarrow\! v_1 \otimes (\downarrow\! w) ^{\otimes (j-1)}) \Big] \otimes (\downarrow\! w)^{\otimes (n-2)-k-(j-1)}\Big)=-\downarrow\! w,\; 1\leq j \leq n-k-1$.  Similarly, there are $(n-k-1)-1 +1 = n-k-1$ such terms in  $\displaystyle\sum_{i+j=n+1} m_i^{\prime}m_j^{\prime} (\downarrow\! v_1 \otimes (\downarrow\! w)^{\otimes k} \otimes \downarrow\! v_1 \otimes (\downarrow\! w)^{\otimes (n-2)-k})$.\\

$\Rightarrow \displaystyle\sum_{i+j=n+1} m_i^{\prime}m_j^{\prime} (\downarrow\! v_1 \otimes (\downarrow\! w)^{\otimes k} \otimes \downarrow\! v_1 \otimes (\downarrow\! w)^{\otimes (n-2)-k}) = (n-k-1) \downarrow\! w - (n-k-1) \downarrow\! w = 0$.\\

\noindent $\bullet$ \textbf{Case 2: $v=v_2$}.  Let us consider $m_i^{\prime}m_j^{\prime}(\downarrow\! v_1 \otimes (\downarrow\! w)^{\otimes k} \otimes \downarrow\! v_2 \otimes (\downarrow\! w)^{\otimes (n-2)-k}),\; 0 \leq k \leq n-2$.\\
Similarly, to produce $(*)$, $m_j^{\prime}$ must `catch' $(1)$ both $\downarrow\! v_1$ and $\downarrow\! v_2$, or $(2)$ only $\downarrow\! v_2$.\\

\noindent For $(1)$, the only nontrivial way to do this yields:\\
$$m_{n-k-1}^{\prime} (m_{k+2}^{\prime}(\downarrow\! v_1 \otimes (\downarrow\! w) ^{\otimes k} \otimes \downarrow\! v_2) \otimes (\downarrow\! w) ^{\otimes (n-2)-k})= \downarrow\! w$$

\noindent and for $(2)$, the only nontrivial way to do this yields:\\
$$(-1)^{|v_1| + k|w| -(k+1)} m_n^{\prime} (\downarrow\! v_1 \otimes (\downarrow\! w)^{\otimes k} \otimes m_1^{\prime} (\downarrow\! v_2) \otimes (\downarrow\! w)^{\otimes (n-2)-k} ) = -\downarrow\! w$$

$\Rightarrow \displaystyle\sum_{i+j=n+1} m_i^{\prime}m_j^{\prime} (\downarrow\! v_1 \otimes (\downarrow\! w)^{\otimes k} \otimes \downarrow\! v_1 \otimes (\downarrow\! w)^{\otimes (n-2)-k}) = \downarrow\! w - \downarrow\! w = 0$.\\

In either case, if $m_i^{\prime}m_j^{\prime} (\downarrow\! x_1 \otimes \cdots  \otimes\downarrow\! x_{n})$ produces $\downarrow\! w$'s, then 
$$\displaystyle\sum_{i+j=n+1} m_i^{\prime}m_j^{\prime} (\downarrow\! x_1 \otimes \cdots  \otimes\downarrow\! x_{n})=0.$$
\bigskip

We now consider the manner in which $m_i^{\prime}m_j^{\prime} (\downarrow\! x_1 \otimes \cdots  \otimes\downarrow\! x_{n})$ yields a $\downarrow\! v_1$:\\

By defintion of $m_n^{\prime}$, $\downarrow\! v_1$ must be produced by either $m_i^{\prime}(\downarrow\! v_1 \otimes (\downarrow\! w)^{\otimes k} \otimes \downarrow\! v_1 \otimes (\downarrow\!w)^{\otimes (i-2)-k})$ or $m_i^{\prime}(\downarrow\! v_1 \otimes (\downarrow\! w)^{\otimes (i-2)} \otimes \downarrow\! v_2)$.\\

\noindent $\bullet$ \textbf{Case 1:} $\downarrow\! v_1$ is produced by $m_i^{\prime}(\downarrow\! v_1 \otimes (\downarrow\! w)^{\otimes k} \otimes \downarrow\! v_1 \otimes (\downarrow\! w)^{\otimes (i-2)-k})$.\\

We examine the 4 different possibilities for which $m_j^{\prime}$ can yield this arrangement:\\
\begin{center}
$(i)$ $m_j^{\prime}$ produces  the first $\downarrow\! v_1$. $(ii)$ $m_j^{\prime}$ produces a $\downarrow\! w$ in $(\downarrow\! w)^{\otimes k}$.\\
 $(iii)$ $m_j^{\prime}$ produces the second $\downarrow\! v_1$. $(iv)$ $m_j^{\prime}$ produces a $\downarrow\! w$ in  $(\downarrow\! w)^{\otimes (i-2)-k}$.
\end{center}

A key observation to make here is that $(i),\;(ii),\;(iii),$ and $(iv)$ imply that the original arrangement $\downarrow\! x_1 \otimes \cdots \otimes \downarrow\! x_n$ must contain \underline{exactly} 3  $v$'s, once again with $x_1 = v_1$.  This yields 4 subcases:

$\circ$  \textit{Subcase 1}: We have $m_i^{\prime} m_j^{\prime}(\downarrow\! v_1 \otimes (\downarrow\! w)^{\otimes k} \otimes \downarrow\! v_1 \otimes (\downarrow\! w) ^{\otimes l} \otimes \downarrow\! v_1 \otimes (\downarrow\! w)^{\otimes n-k-l-3})$:

$\circ$ \textit{Subcase 2}: We have $m_i^{\prime} m_j^{\prime}(\downarrow\! v_1 \otimes (\downarrow\! w)^{\otimes k} \otimes \downarrow\! v_1 \otimes (\downarrow\! w) ^{\otimes l} \otimes \downarrow\! v_2 \otimes (\downarrow\! w)^{\otimes n-k-l-3})$:

$\circ$ \textit{Subcase 3}: We have $m_i^{\prime} m_j^{\prime}(\downarrow\! v_1 \otimes (\downarrow\! w)^{\otimes k} \otimes \downarrow\! v_2 \otimes (\downarrow\! w) ^{\otimes l} \otimes \downarrow\! v_1 \otimes (\downarrow\! w)^{\otimes n-k-l-3})$:

$\circ$ \textit{Subcase 4}: We have $m_i^{\prime} m_j^{\prime}(\downarrow\! v_1 \otimes (\downarrow\! w)^{\otimes k} \otimes \downarrow\! v_2 \otimes (\downarrow\! w) ^{\otimes l} \otimes \downarrow\! v_2 \otimes (\downarrow\! w)^{\otimes n-k-l-3})$:

Let us consider subcase 1:

\noindent $(i)\;m_j^{\prime}$ must take the first two $\downarrow\! v_1$'s.  We have:
$$m_i^{\prime} \Big(\Big[m_j^{\prime}(\downarrow\! v_1 \otimes (\downarrow\! w)^{\otimes k} \otimes \downarrow\! v_1 \otimes (\downarrow\! w)^{\otimes j-k-2}) \otimes (\downarrow\! w)^{\otimes l-(j-k-2)} \Big] \otimes \downarrow\! v_1 \otimes (\downarrow\! w)^{\otimes n-k-l-3}\Big) = \downarrow\! v_1 $$
Now $k+2 \leq j \leq l+k+2$, so there are $(l+k+2)-(k+2)+1 = l+1$ such terms.\\

\noindent $(ii)\; m_j^{\prime}$ must take only the second $\downarrow\! v_1$.  We have:
$$(-1)^{|v_1|+k|w|-(k+1)} m_i^{\prime} \Big(\! \downarrow\! v_1 \otimes (\!\downarrow\!\! w)^{\otimes k} \otimes \!\Big[m_j^{\prime}(\downarrow\! v_1\otimes (\downarrow\! w)^{\otimes (j-1)}) \otimes\! (\!\downarrow\!\!w)^{\otimes l-(j-1)} \!\Big] \otimes \downarrow\! v_1 \otimes (\downarrow\! w)^{\otimes n-k-l-3}\Big) \!=\! -\!\downarrow\! v_1$$
Now $1\leq j \leq l+1$, so there are $(l+1)-1+1= l+1$ such terms.\\

\noindent $(iii)\; m_j^{\prime}$ must take the second and third $\downarrow\! v_1$'s.  We have:
$$(-1)^{|v_1|+k|w|-(k+1)} m_i^{\prime} \Big(\! \downarrow\! v_1 \otimes (\!\downarrow\! w)^{\otimes k} \otimes \Big[\!m_j^{\prime}(\downarrow\! v_1\otimes (\downarrow\! w)^{\otimes l} \otimes \downarrow\! v_1 \otimes (\downarrow\! w)^{\otimes j-l-2}) \Big] \otimes (\downarrow\! w)^{\otimes n-k-j+1}\Big) \!=\! -\downarrow\! v_1$$
Now $l+2\leq j \leq n-k-1$, so there are $(n-k-1)-(l+2)+1= n-k-l-2$ such terms.\\

\noindent $(iv)\; m_j^{\prime}$ must take only the third $\downarrow\! v_1$.  We have:
$$(-1)^{2|v_1|+(k+l)|w|-(k+l+2)} m_i^{\prime} \Big(\! \downarrow\! v_1 \otimes (\!\downarrow\! w)^{\otimes k} \otimes \downarrow\! v_1 \otimes (\downarrow\! w)^{\otimes l} \otimes \Big[\!m_j^{\prime}(\downarrow\! v_1 \otimes (\downarrow\! w)^{\otimes (j-1)}) \otimes (\downarrow\! w)^{\otimes n-k-l-j-2} \Big] \Big) \!=\! \downarrow\! v_1$$
Now $1\leq j \leq n-k-l-2$, so there are $(n-k-l-2)-1+1= n-k-l-2$ such terms.\\

$\Rightarrow \displaystyle\sum_{i+j=n+1} m_i^{\prime}m_j^{\prime} (\downarrow\! v_1 \otimes (\downarrow\! w)^{\otimes k} \otimes \downarrow\! v_1 \otimes (\downarrow\! w) ^{\otimes l} \otimes \downarrow\! v_1 \otimes (\downarrow\! w)^{\otimes n-k-l-3}) = (l+1) \downarrow\! v_1 -(l+1)\downarrow\! v_1 -(n-k-l-2)\downarrow\! v_1 + (n-k-l-2) \downarrow\! v_1 = 0$.\\

A similar argument holds for subcases 2, 3, and 4.  Hence, our result holds for case 1.\\

\noindent $\bullet$ \textbf{Case 2:} $\downarrow\! v_1$ is produced by $m_i^{\prime}(\downarrow\! v_1 \otimes (\downarrow\! w)^{\otimes (i-2)} \otimes \downarrow\! v_2)$.\\

We examine the 2 different possibilities for which $m_j^{\prime}$ can yield this arrangement:\\
\begin{center}
$(i)$ $m_j^{\prime}$ produces  the $\downarrow\! v_1$. \\
$(ii)$ $m_j^{\prime}$ produces a $\downarrow\! w$ in $(\downarrow\! w)^{\otimes (i-2)}$.\\
\end{center}

A similar observation to case 1 can be made here regarding the original arrangement $\downarrow\! x_1 \otimes \cdots \otimes \downarrow\! x_n$ containing \underline{exactly} 3  $v$'s, once again with $x_1 = v_1$.  In this case, $x_n=v_2$.  This yields 2 subcases:

$\circ$ \textit{Subcase 1}: We have $m_i^{\prime} m_j^{\prime}(\downarrow\! v_1 \otimes (\downarrow\! w)^{\otimes k} \otimes \downarrow\! v_1 \otimes (\downarrow\! w) ^{\otimes (n-k-3)} \otimes \downarrow\! v_2)$

$\circ$ \textit{Subcase 2}: We have $m_i^{\prime} m_j^{\prime}(\downarrow\! v_1 \otimes (\downarrow\! w)^{\otimes k} \otimes \downarrow\! v_2 \otimes (\downarrow\! w) ^{\otimes (n-k-3)} \otimes \downarrow\! v_2)$

Let us consider subcase 1:

\noindent $(i)\; m_j^{\prime}$ must take both $\downarrow\! v_1$'s. We have:
$$m_i^{\prime} \Big( \Big[ m_j^{\prime}(\downarrow\! v_1 \otimes (\downarrow\! w)^{\otimes k} \otimes \downarrow\! v_1 \otimes (\downarrow\! w)^{\otimes j-k-2}) \Big] \otimes (\downarrow\! w)^{\otimes n-j-1} \otimes \downarrow\! v_2 \Big) = \downarrow\! v_1$$
Now $k+2\leq j \leq n-1$, so there are $(n-1)-(k+2)+1= n-k-2$ such terms.\\

\noindent $(ii)\; m_j^{\prime}$ must take the second $\downarrow\! v_1$ only.  We have:
$$(-1)^{|v_1|+k|w|-(k+1)} m_{i}^{\prime} \Big( \downarrow\! v_1 \otimes (\downarrow\! w)^{\otimes k} \otimes \Big[m_{j}^{\prime}(\downarrow\! v_1 \otimes (\downarrow\! w)^{\otimes j-1}) \otimes (\downarrow\! w)^{\otimes n-k-j-2} \Big] \otimes \downarrow\! v_2 \Big) = -\downarrow\! v_1$$
Now $1\leq j \leq n-k-2$, so there are $(n-k-2)-(1)+1= n-k-2$ such terms.\\
\noindent This implies that

\noindent $\sum_{i+j=n+1} m_i^{\prime}m_j^{\prime} (\downarrow\! v_1 \otimes (\downarrow\! w)^{\otimes k} \otimes \downarrow\! v_1 \otimes (\downarrow\! w) ^{\otimes (n-k-3)} \otimes \downarrow\! v_2) = (n-k-2)\downarrow\! v_1 -  (n-k-2)\downarrow\! v_1 = 0$.\\

A similar argument may be made for subcase 2.  Hence, our result holds for case 2.\\

So $\displaystyle\sum_{i+j=n+1} m_i^{\prime}m_j^{\prime} (\downarrow\! x_1 \otimes \cdots  \otimes\downarrow\! x_{n})=0,\; \forall \; x_1 \cdots x_n \; \in \; V$.\\

Thus $D^2(\downarrow\! x_1 \otimes \cdots  \otimes\downarrow\! x_{n}) = 0$\\

By induction, $D^2 = 0$ on any number of inputs.\\

Hence the preceding maps $m_n$ defined on the graded vector space $V$ form an $A_{\infty}$ algebra structure.
\end{proof}

\section{Induced $L_{\infty}$ Algebra}
The $A_\infty$ algebra structure on $V=V_0\oplus V_1$ that was constructed in this note can be skew symmetrized  to yield an $L_\infty$ algebra structure on $V$; see Theorem 3.1 in \cite{LM} for details.  This $L_\infty$ algebra will thus join the collection of previously defined such structures on $V$.  The relationship among these algebras will be a topic for future research.

\end{document}